\documentclass[reqno]{amsart} 
\usepackage{pgfplots}

\usepackage{color}
\usepackage{xcolor}

\usepackage{ifpdf}
\ifpdf 
\usepackage{graphicx,import}   
\usepackage[pdftex,     
plainpages=false,   
breaklinks=true,    
colorlinks=true,
linkcolor=blue,
 citecolor=green,
pdftitle={Uniqueness of Lipschitz transport by non-autonomous fields},
pdfauthor={Paolo Bonicatto, Giacomo Del Nin}
]{hyperref} 
\else 
\usepackage{graphicx}       
\usepackage{hyperref}       
\fi 

\usepackage{amsfonts,amsmath}	
\usepackage{amssymb}
\usepackage{verbatim}
\usepackage{amsopn}
\usepackage[english]{babel}
\usepackage{amsthm}
\usepackage{enumerate}
\usepackage{mathrsfs}	
\usepackage{mathtools}
\usepackage{harpoon}
\usepackage{esint}
\usepackage{todonotes}

\usepackage{stmaryrd}
\usepackage{bm}
\usepackage{bbm}
\usepackage{caption}
\usepackage{subcaption}
\usepackage{nicefrac}
\usepackage[a4paper,left=35mm,right=35mm,top=34mm,bottom=40mm,marginpar=25mm]{geometry}

\newcommand{\Crm}{\mathrm{C}}

\newcommand{\Mrm}{\mathrm{M}}
\newcommand{\Nrm}{\mathrm{N}}

\newcommand{\Ccal}{\mathcal{C}}

\newcommand{\Lcal}{\mathcal{L}}

\newcommand{\Scal}{\mathcal{S}}

\newcommand{\Mbf}{\mathbf{M}}

\newcommand{\Lscr}{\mathscr{L}}

\newcommand{\altnorm}[1]{{\left\vert\kern-0.25ex\left\vert\kern-0.25ex\left\vert #1 \right\vert\kern-0.25ex\right\vert\kern-0.25ex\right\vert}}
\newcommand{\eps}{\varepsilon}

\newcommand{\dprb}[1]{\bigl\langle #1 \bigr\rangle}

\newcommand{\tbf}{\mathbf{t}}
\newcommand{\pbf}{\mathbf{p}}

\theoremstyle{definition} \newtheorem{definition}{Definition}[section]
\theoremstyle{definition} \newtheorem{remark}[definition]{Remark}
\theoremstyle{plain} \newtheorem{lemma}[definition]{Lemma}
\theoremstyle{plain} \newtheorem{proposition}[definition]{Proposition}

\theoremstyle{plain} \newtheorem{theorem}[definition]{Theorem}
\newtheorem*{theorem*}{Theorem}
\newtheorem*{proposition*}{Proposition}
\theoremstyle{plain} 
\theoremstyle{definition} 
\theoremstyle{plain} 
\theoremstyle{definition} 
\theoremstyle{plain} 
\theoremstyle{plain}

\DeclareMathOperator{\dive}{div}

\DeclareMathOperator{\dist}{dist}

\DeclareMathOperator{\Lip}{Lip}

\DeclareMathOperator{\Wedge}{{\textstyle\bigwedge}}

\DeclareMathOperator*{\esssup}{esssup}

\newcommand{\ee}{\mathrm{e}}

\newcommand{\sbullet}{\begin{picture}(1,1)(-0.5,-2.5)\circle*{2}\end{picture}}
\newcommand{\frarg}{\,\sbullet\,}

\newcommand{\R}{\mathbb{R}}

\newcommand{\N}{\mathbb{N}}

\newcommand{\e}{\varepsilon}

\renewcommand{\L}{\mathscr L}

\newcommand{\dd}{\mathrm{d}}

\newcounter{counter}

\newcommand{\Dscr}{\mathscr{D}}

\renewcommand{\L}{\mathscr L}

\theoremstyle{plain} 
\theoremstyle{plain} \newtheorem*{mthm*}{Main Theorem}
\theoremstyle{plain} \newtheorem*{conjecture*}{Conjecture}
\theoremstyle{plain} 
\theoremstyle{plain} \newtheorem*{problem*}{Problem}

\newcommand{\bbb}[1]{\llbracket #1 \rrbracket}

\numberwithin{equation}{section}

\usepackage{color}
\usepackage{graphicx}
\usepackage{tikz}
\usepackage{framed}
\definecolor{shadecolor}{rgb}{0.94, 0.97, 1.0}

\usepackage{pict2e}
\makeatletter
\DeclareRobustCommand{\intprod}{%
  \mathbin{\mathpalette\int@prod{(0.1,0)(0.9,0)(0.9,0.8)}}}
\DeclareRobustCommand{\restrict}{%
  \mathbin{\mathpalette\int@prod{(0.1,0.8)(0.1,0)(0.9,0)}}}	
\newcommand{\int@prod}[2]{%
  \begingroup
  \sbox\z@{$\m@th#1+$}%
  \setlength\unitlength{\wd\z@}%
  \begin{picture}(1,1)
  \roundcap
  \polyline#2
  \end{picture}%
  \endgroup
}
\makeatother

\author{Paolo Bonicatto}
\address[P.\ Bonicatto]{Università di Trento,
Dipartimento di Matematica,
Via Sommarive 14, 38123 Trento, Italy}
\email{paolo.bonicatto@unitn.it}

\author{Giacomo Del Nin}
\address[G.\ Del Nin]{Max Planck Institute for Mathematics in the Sciences, Inselstrasse 22, 04103 Leipzig, Germany}
\email{giacomo.delnin@mis.mpg.de}

\title[Uniqueness of time-dependent transport of currents]{Well-posedness of the transport of normal currents\\ by time-dependent vector fields}

\date{\today}

\begin{document}

\begin{abstract}
    We prove existence and uniqueness for the transport equation for currents (Geometric Transport Equation) when the driving vector field is time-dependent, Lipschitz in space and merely integrable in time. This extends previous work where well-posedness was shown in the case of a time-independent, Lipschitz vector field. The proof relies on the decomposability bundle and requires to extend some of its properties to the class of functions that in one direction are only absolutely continuous, rather than Lipschitz.
	\end{abstract}

\maketitle

\section{Introduction}
Let $b:[0,1]\times\R^d\to \R^d$ be a Borel vector field. For every $t \in [0,1]$ we write $b_t(x):=b(t,x)$.
In \cite{BDR} the authors considered the Geometric Transport Equation 
\begin{equation}\label{eq:GTE}
    \frac{\dd}{\dd t} T_t+\Lcal_{b_t}T_t=0\tag{GTE}
\end{equation}
and studied several of its properties.
This equation, which we understand in the distributional sense (see \eqref{eq:def_weak_sense}), describes the motion of $k$-currents $T_t$ driven by the given vector field $b$, for any $k\in \{0,\ldots,d\}$. We refer the reader to \cite{BDR} for some motivating examples and for the connection of \eqref{eq:GTE} with the transport and continuity equations and several other PDEs from mathematical physics. 

A natural question is to understand under which regularity assumptions on $b$ and/or on the currents $T_t$ the initial value problem associated with \eqref{eq:GTE} is well posed, i.e., when for every $\bar T$ there exists a unique solution starting from $\bar T$ at time $0$. In \cite{BDR-Lipschitz} it was proven that well-posedness holds in the class of normal $k$-currents $\Nrm_k(\R^d)$ as soon as the vector field $b$ is Lipschitz and autonomous, but extending this to time-dependent vector fields was left open. 

In order to understand which regularity is natural to require on $b$, we recall that \eqref{eq:GTE} corresponds, in the case of $0$-currents, to the continuity equation
\[
\frac{\dd}{\dd t}\mu_t +\dive(b_t\mu_t)=0
\]
for a family of (possibly signed) measures $\mu_t$.
For this equation several well-posedness results are already established (see, for instance, \cite[Chapter 8]{AGS}), and a natural class in this case is given by $L^1_t \Lip_x$, namely those vector fields that satisfy 
\begin{equation}\label{eq:Lip_assumption}\tag{L}
\int_0^1 \big(\|b_t\|_{C^0(\R^d)}+\Lip(b_t;\R^d)\big)\,\dd t<+\infty,
\end{equation}
where $\Lip(b_t;\R^d)$ denotes the Lipschitz constant of $b_t$.  
In the present work we extend the well-posedness result of \cite{BDR-Lipschitz} to the class of all vector fields satisfying \eqref{eq:Lip_assumption}. We summarize this in the following statement (see Theorem \ref{thm:existence} and Theorem \ref{thm:uniqueness}).

\begin{theorem*}
    Let $b:[0,1]\times\R^d\to\R^d$ be a vector field satisfying \eqref{eq:Lip_assumption}.
    Moreover, let $\bar T\in \Nrm_k(\R^d)$. Then there exists a unique family of currents $T_t\in L^\infty([0,1]; \Nrm_k(\R^d))$ that solves
\[
\begin{cases}
    \frac{\dd}{\dd t} T_t+\Lcal_{b_t}T_t=0\\
    T_0=\bar T.
\end{cases}
\]
More precisely, denoting by $(\Phi^0_t)_t$ the flow of $b$ starting from time $0$, the solution is given by the pushforward $T_t=(\Phi^0_t)_* \bar T$ for every $t \in [0,1]$.
\end{theorem*}

As in \cite{BDR-Lipschitz}, the strategy to prove this result heavily relies on the decomposability bundle $V(\mu,\frarg)$ associated with a measure $\mu$ in $\R^d$, see \cite{AM}. A central role in the proof is played by the space-time flow $\Psi(t,x)=(t,\Phi^0(t,x))$ associated with $b$ (see Subsection \ref{subsec:flows}).
Using the decomposability bundle we are able to deduce some pointwise formulae for the currents $T_t$ involving the derivatives of $\Psi$.  
However, the difficulty of directly using the same approach as in \cite{BDR-Lipschitz} lies in the fact that the flow associated with a vector field satisfying \eqref{eq:Lip_assumption} is still Lipschitz in space but only absolutely continuous in time (we denote this space of functions by $AC_t\Lip_x$, see Definition \ref{def:AC_t_Lip_x}), while the original formulation of the decomposability bundle only allows for Lipschitz functions. For this reason, in the first part of the paper we will focus on extending some differentiability results to the space $AC_t\Lip_x$. 
We mention here one of the results as an example, see Theorem \ref{thm:differentiability}.
\begin{theorem*}
    Let $f\in AC_t\Lip_x(\R\times\R^d)$, and let $\mu$ be a measure in $\R\times\R^d$ that can be written as $\mu=\L^1(\dd t ) \otimes \mu_t$ for some family of measures $\mu_t$. Then $f$ is differentiable along $V(\mu,(t,x))$ for $\mu$-a.e. $(t,x)$.
\end{theorem*}

To the best of our knowledge, this is the first extension of \cite{AM} to the non-Lipschitz setting.
This result, together with others in Section \ref{sec:bundle}, allows, for instance, to give a meaning to the pushforward of normal currents with respect to the space-time flow map $\Psi$, and this is the first stepping stone through which we develop our proof of the well-posedness of \eqref{eq:GTE}.

Finally, we comment on the requirement that the currents be normal. One may ask for instance what happens if the currents are only required to have (integrable in time) finite mass. In this case the first difficulty is that a priori we are not even able to give a meaning to the distributional formulation of \eqref{eq:GTE}. Nevertheless, we show by means of a simple example (see Section \ref{sec:finite_mass}) that if one could come up with any notion of solution satisfying two very natural requirements (namely \textit{stability} and \textit{consistency} with the smooth case, see Definition \ref{def:natural_family}), then there would exist two different finite mass solutions starting from the same initial datum. The vector field in this example is actually autonomous and Lipschitz, 
thus suggesting that the class of normal currents is the right setting to obtain well-posedness for the problem. \\

We close the Introduction by giving a summary of how the sections are organized. In Section \ref{sec:prelims} we recall some preliminaries and prove some lemmas concerning the decomposability bundle, the maximal function and flows of vector fields. In Section \ref{sec:bundle} we introduce the class $AC_t\Lip_x$ and we extend some results concerning the decomposability bundle to this class. In Section \ref{sec:existence} we show existence of a solution under the assumption \eqref{eq:Lip_assumption}. In Section \ref{sec:uniqueness} we show uniqueness of the solution under the assumption \eqref{eq:Lip_assumption}. Finally in Section \ref{sec:finite_mass} we give an example of non-uniqueness in the class of currents with finite mass.

\section{Preliminaries and notation}\label{sec:prelims}

Let $d \in \N$ be the ambient dimension. We will often use the projection maps $\tbf \colon\R\times \R^d\to \R$, $\pbf \colon\R\times \R^d\to \R^d$ from the (Euclidean) space-time $\R\times\R^d$ onto the time and space variables, respectively, which are given as
	\begin{equation}\label{eq:projections_def}
		\tbf(t,x) = t,  \qquad
		\pbf(t,x) = x.
	\end{equation}
	We also define, for every given $t\in\R$, the immersion map $\iota_t\colon \R^d \to \R\times\R^d$ by
	\[
	\iota_t(x):=(t,x),  \qquad (t,x) \in \R\times \R^d.
	\]
For the notation and the definitions of multilinear algebra and of currents we refer the reader to \cite{KrantzParks08book} and also \cite{BDR,BDR-Lipschitz}. 

Here we only recall the distributional formulation of the equation \eqref{eq:GTE}. If $b$ is a Borel vector field and $T_t$ is a family of normal currents in $\R^d$ such that 
\begin{equation*}
\int_0^1 |b(t,x)| \,\dd(\|T_t\|+\|\partial T_t \|)(x)\, \dd t < \infty  
\end{equation*}
then we say that $(T_t)_t$ is a \textit{weak solution} to \eqref{eq:GTE} if 
\begin{equation}\label{eq:def_weak_sense}
\int_0^1 \langle T_t, \omega\rangle \psi'(t) - \langle \Lcal_{b_t}T_t,\omega \rangle \psi(t) \,\dd t  = 0 
\end{equation}
for every $\omega \in \Dscr^k(\R^d)$ and every $\psi \in \Crm_c^\infty((0,1))$ where 
\[
\Lcal_{b_t} T_t := - b_t \wedge \partial T_t - \partial (b_t \wedge T_t).
\]

\subsection{A lemma on differentiation}

We will need the following elementary result: 

\begin{lemma}\label{lemma:chain_rule}
    Let $f:\R^n\to\R^m$ be differentiable in direction $v$ at $x\in\R^n$, and let $g:\R^m\to\R^p$ be differentiable in direction $df|_x [v]$ at $f(x)$. Assume that one of the following conditions holds:
    \begin{enumerate}
        \item either $f$ or $g$ is a linear map;
        \item $f$ and $g$ are Lipschitz.
    \end{enumerate}
    Then $g\circ f$ is differentiable in direction $v$, and satisfies the chain rule
    \[
    d(g\circ f)|_x[v]=dg|_{f(x)} [df|_x [v]].
    \]
\end{lemma}

\begin{proof}
    We only prove (2), as the proof of (1) is even simpler. For every $h\in \R$ we can write  
    \begin{align*}
        g(f(x+h df|_x[v]))-g(f(x))&=g(f(x)+h df|_x [v]+o(h))-g(f(x))\\
        &= g(f(x)+h df|_x[v])+o(h)-g(f(x))\\
        &=h dg|_{f(x)}[df|_x[v]]+o(h).
    \end{align*}
    This proves the desired claim.
\end{proof}

\subsection{Decomposability bundle}
We recall the concept of decomposability bundle, introduced in \cite{AM}, and some of its properties. Given a locally finite Borel measure $\mu$ in $\R^n$, the decomposability bundle is a map $x\mapsto V(\mu,x)$ that to $\mu$-a.e. point $x$ associates a linear subspace $V(\mu,x)$ that is maximal (in a suitable way) with respect to the following property:

\begin{theorem}[{\cite[Theorem~1.1]{AM}}]\label{thm:bundle_AM}
    Let $\mu$ be a locally finite measure in $\R^n$ and let $f:\R^n\to\R^m$ be a Lipschitz function. Then $f$ is differentiable along $V(\mu,x)$ at $\mu$-a.e. $x$. 
\end{theorem}

In the following we will denote by $d_V$ the restriction of the differential to $V(\mu,x)$. Moreover, if $T$ is a normal current, we denote by $d_T$ the quantity $d_V$ where $V=V(\|T\|,x)$. 
Below we rephrase the statement that can be found in \cite[Proposition~5.17]{AM}, with slightly different assumptions. We require the map $f$ to be proper, instead of assuming that the current has compact support, but the two statements are virtually the same. Moreover we write $\langle \omega(f(x)),\Wedge^k d_T f[\tau(x)]\rangle$ in place of $\langle (f_T^\#)\omega(x),\tau(x)\rangle$.

\begin{proposition}[Pushforward {\cite[Proposition~5.17]{AM}}]
Let $T=\tau\mu$ be a normal $k$-current in $\R^n$, and let $f:\R^n\to\R^n$ be a proper Lipschitz map. Then the pushforward $f_*T$ admits the representation
\begin{equation}\label{eq:pushforward_lipschitz}
    \langle f_*T,\omega\rangle=\int \langle \omega(f(x)),\Wedge^kd_Tf [\tau(x)] \rangle\,\dd\|T\|(x)\qquad\text{for every $\omega\in \Dscr^k(\R^n)$.}
\end{equation}
\end{proposition}
We also recall the following lemma, whose proof can be found in \cite[Lemma 2.1]{BDR-Lipschitz}: 
\begin{lemma}\label{lemma:wedge_perp}\phantom{.}
		\begin{enumerate}
			\item[(i)] Let $\tau\in\Wedge_1 \R^n$, $\sigma\in\Wedge_k(\R^n)$ and $\alpha\in\Wedge^1(\R^n)$, $\beta\in\Wedge^k(\R^n)$. Suppose that $\dprb{ v,\alpha}=0$ for every $v\in \mathrm{span}(\sigma)$. Then
			\[
			\dprb{ \tau\wedge\sigma,\alpha\wedge\beta}=\dprb{\tau,\alpha} \dprb{ \sigma,\beta}.
			\]
			\item[(ii)] Let $\tau\in\Wedge_1 \R^n$, $\sigma\in\Wedge_k(\R^n)$ and $\alpha\in\Wedge^{k+1}(\R^n)$. Suppose that $\tau\perp \mathrm{span}(\alpha)$.
			Then,
			\[
			\dprb{ \tau\wedge\sigma,\alpha}=0.
			\]
		\end{enumerate}
	\end{lemma}

\subsection{Maximal function}\label{sec:maximal}
Given a function $f\in L^1(\R^n)$ we define the \emph{uncentered maximal function} of $f$ as
\[
M f(x) := \sup\left\{\fint_{B} |f(y)|\, \dd y: B \text{ is a ball with } x \in B\right\}.
\]
We recall the following weak $(1,1)$ estimate: there exists a dimensional constant $C>0$ such that
\begin{equation}\label{eq:weak1_maximal_function}
|\{x:Mf(x)>\lambda\}|\le \frac{C}{\lambda}\|f\|_{L^1(\R^n)},
\end{equation}
see, e.g., \cite[Theorem 3.17]{folland} and also \cite[Exercise 23, Chapter 3]{folland}. 
We recall that if $f \in AC(\R)$ then $f$ is differentiable $\L^1$-a.e.; if $g$ is an $L^1$ function such that $|f'(x)|\le g(x)$ for $\L^1$-a.e. $x$ then the restriction of $f$ to the set $\{x:Mg(x)\le \lambda\}$ is $\lambda$-Lipschitz.

\subsection{Flows}\label{subsec:flows}
In this paragraph we recall some results on the Cauchy-Lipschitz theory for vector fields satisfying \eqref{eq:Lip_assumption}. 
By \cite[Lemma 8.1.4]{AGS}, given such a field $b \colon [0,1] \times \R^d \to \R^d$, there exists a unique flow map for $b$, namely there exists a unique function $\Phi \colon [0,1] \times [0,1] \times \R^d \to \R^d$ such that for every $s\in [0,1]$ and for every $x \in \R^d$ the curve $t \mapsto \Phi(t,s,x)=:\Phi_t^s(x)$ is the unique solution to the ODE driven by $b$ starting at time $s$ from the point $x$, i.e. 
\[
\begin{cases}
    \frac{d}{dt}\Phi_t^s(x)=b_t(\Phi_t^s(x))\\
    \Phi_s^s(x)=x
\end{cases}.
\]
The ordinary differential equation is intended in the integral sense, namely for every $s,t\in[0,1]$ and $x\in\R^d$
\[
\Phi_t^s(x) = x + \int_s^t b_{r}(\Phi_r^s(x))\,\dd r. 
\]
As an easy consequence of the uniqueness of the trajectories we have the validity of the following semigroup-like property: 
\begin{equation}\label{eq:semigroup-like}
\forall s,s'\in[0,1], \, \forall x \in \R^d: \qquad \Phi_t^s(\Phi_s^{s'}(x)) = \Phi_t^{s'}(x).
\end{equation}
We define for later use the map $\Psi \colon [0,1]\times \R^d \to \R^d$ by setting 
\begin{equation}\label{eq:def_Psi}
\Psi(t,x):=(t,\Phi_t^0(x)). 
\end{equation}
Observe that
\[
\Psi^{-1}(s,y)=(s,(\Phi_s^0)^{-1}(y)).
\]

The following lemma contains some regularity properties of the flow map. 

\begin{lemma}[The flow is $AC_t\Lip_x$]\label{lemma:gronwall}
Let $b:[0,1]\times\R^d\to\R^d$ satisfy \eqref{eq:Lip_assumption}.
Then for every $s\in[0,1]$ the flow map $(t,x)\mapsto \Phi^s_t(x)$ is uniquely defined and belongs to $AC_t\Lip_x([0,1]\times\R^d)$, and more precisely for every $t_1<t_2$ and $x,y\in\R^d$ it holds
\[
|\Phi^s_{t_1}(x)-\Phi^s_{t_2}(y)|\le \exp\left(\int_{s}^{t_1} \Lip(b_r)\,\dd r\right)|x-y|+\int_{t_1}^{t_2} \|b_r\|_{C^0(\R^d)}\,\dd r.
\]
\end{lemma}

\begin{proof} 
Fix $s \in \R$ and $x \in \R^d$. Then by the definition of integral curve we have 
\[
\Phi_t^s(x) = x + \int_s^t b_{r}(\Phi_r^s(x))\, \dd r. 
\]
This implies that if $t_1<t_2$ it holds 
\begin{equation}\label{eq:ACLip_estimates_flow_1}
|\Phi_{t_1}^s(x) - \Phi_{t_2}^s(x)| \le \int_{t_1}^{t_2} |b_r(\Phi^s_r(x))| \,\dd r \le \int_{t_1}^{t_2} \|b_r\|_{C^0(\R^d)} \,\dd r.
\end{equation}
Fix now $t\in[0,1]$ and $x,y\in\R^d$. By using the definition of integral curve we have 
    \begin{align*}
    |\Phi_t^s(x)-\Phi_t^s(y)| & = \left| x+\int_s^t b_r(\Phi_r^s(x)) \, \dd r - 
    y-\int_s^t b_r(\Phi_r^s(y)) \, \dd r \right| \\ 
    & \le |x-y| + \int_s^t | b_r(\Phi_r^s(x))-b_r(\Phi_r^s(y))| \, \dd r\\
    & \le |x-y| + \int_s^t \Lip(b_r) |\Phi_r^s(x)-\Phi_r^s(y)| \, \dd r.
    \end{align*}
    By applying the integral form of Gronwall inequality we conclude that 
    \begin{equation}\label{eq:ACLip_estimates_flow_2}
    |\Phi_t^s(x)-\Phi_t^s(y)| \le \exp{\left(\int_s^t \Lip(b_r)\, \dd r\right)} |x-y|. 
    \end{equation}
The conclusion follows by \eqref{eq:ACLip_estimates_flow_1} and \eqref{eq:ACLip_estimates_flow_2} choosing $t=t_1$.
\end{proof}

We now show that under the assumption \eqref{eq:Lip_assumption} we can find a ‘‘universal" set of times at which the vector field $b$ has the Lebesgue property at every point in space.

\begin{lemma}\label{lemma:uniform_Leb_points} Let $b\colon [0,1]\times \R^d \to \R^d$ satisfy \eqref{eq:Lip_assumption} and let 
\[
G:=\{(t,x) \in [0,1] \times \R^d: \, t \text{ is a Lebesgue point of } s \mapsto b_s(x)\}. 
\]
Then there exists a $\L^1$-negligible set $N \subset [0,1]$ such that 
\[
N^c \times \R^d \subset G,
\]
and moreover every $t\in N^c$ is a Lebesgue point of the map $t\mapsto \Lip(b_t)$.
\end{lemma}

\begin{proof}
    Given a set $A \subset \R^{d+1}$ we denote as usual its sections by 
    \[
    A_t :=\{x \in \R^d: (t,x) \in A\}, \qquad  A^x :=\{t \in \R: (t,x) \in A\}. 
    \]
    By Lebesgue's differentiation Theorem, it holds $\L^1((G^x)^c)=0$ for every $x \in \R^d$. In particular, by Fubini-Tonelli,
    \begin{align*}
        0 & = \int_{\R^d} \L^1((G^x)^c)\, \dd x
        = \L^{d+1}(G^c) 
        = \int_{\R} \L^d((G_t)^c)\, \dd t. 
    \end{align*}
    We deduce that there exists a set $N_1$, with $\L^1(N_1)=0$, such that if $t \notin N_1$, it holds $\L^d((G_t)^c)=0$: in particular $G_t$ is dense in $\R^d$ for every $t \notin N_1$.
    Now, in view of the assumptions on the vector field, there is a $\L^1$-negligible subset $N_2 \subset \R$ such that if $t \notin N_2$ then $t$ is a Lebesgue point of $s\mapsto \Lip(b_s)$ - in particular, $\Lip(b_t)<\infty$. 
    We now define $N:=N_1 \cap N_2$. 
    
    We claim that the set $G_t$ is closed for every $t \notin N$. Combining this with the density, we conclude, because we necessarily have that $G_t = \R^d$ for every $t \notin N$.

    In order to show that $G_t$ is closed we can argue as follows: pick $x \notin G_t$ so that we can find $\delta>0$ and a sequence $\eps_j \to 0$ such that 
    \[
    \frac{1}{\eps_j}\int_t^{t+\eps_j} |b_s(x)-b_t(x)|\,\dd s > \delta 
    \]
    for every $j \in \N$. Let $\eta>0$ to be chosen later. For $x' \in B_\eta(x)$ we have 
    \begin{align*}
         \frac{1}{\eps_j}\int_t^{t+\eps_j} |b_s(x')-b_t(x')|\,\dd s & \ge \frac{1}{\eps_j}\int_t^{t+\eps_j} |b_t(x)-b_s(x)| - |b_s(x')-b_s(x)| - |b_t(x')-b_t(x)|\,\dd s \\ 
         & \ge \delta - |x-x'|\left( \Lip(b_t) + \frac{1}{\eps_j}\int_t^{t+\eps_j} \Lip(b_s)\,\dd s\right) \\ 
         & \ge \delta - \eta \left( \Lip(b_t) + \frac{1}{\eps_j}\int_t^{t+\eps_j} \Lip(b_s)\,\dd s\right).
    \end{align*}
    Since $t \notin N$, we have that the second summand is finite (and uniformly bounded in $j$). Thus we can pick $\eta>0$ small enough so that 
\[
\delta - \eta \left( \Lip(b_t) + \frac{1}{\eps_j}\int_t^{t+\eps_j} \Lip(b_s)\,\dd s\right) > \frac{\delta}{2}
\]
which means that every $x' \in B_\eta(x)$ cannot lie in $G_t$, i.e. $(G_t)^c$ is open. This concludes the proof.
\end{proof}

The following lemma establishes some differentiability properties of the flow.
\begin{lemma}\label{lemma:diff_flow} 
 Let $b\colon [0,1]\times \R^d \to \R^d$ satisfy \eqref{eq:Lip_assumption}. Then, it holds: 
 \begin{enumerate}
     \item $D\Psi(t,x)[(1,0)] = (1, b_t(\Phi^0_t(x)))$ for every $t \in (0,1)$ and $x \in \R^d$;
     \item $D\Psi^{-1}(\Psi(t,x))[(1,b_t(\Phi_t^0(x)))]=(1,0)$ for $\L^1$-a.e $t \in [0,1]$ and for every $x \in \R^d$.
\end{enumerate}
\end{lemma}

\begin{proof}
    Point (1) is immediate and follows from the very definition of flow. We omit the proof. 

    For Point (2), we argue as follows. By definition we have 
    \begin{align*}
    D\Psi^{-1}(\Psi(t,x))[(1,b_t(\Phi_t^0(x)))] = & \lim_{h\to 0} \frac{\Psi^{-1}(\Psi(t,x)+h(1,b_t(\Phi_t^0(x))))-\Psi^{-1}(\Psi(t,x))}{h}\\
    =&\lim_{h \to 0}\frac{\Psi^{-1}\Big(\big(t+h,\Phi_t^0(x)+h b_t(\Phi_t^0(x))\big)-(t,x)\Big)}{h}\\
    =&\lim_{h\to 0} \frac{\big(t+h,(\Phi_{t+h}^0)^{-1}\big(\Phi_t^0(x)+h b_t(\Phi_t^0(x))\big)\big)-(t,x)}{h}\\
    =& \left( 1,\lim_{h \to 0}\frac{(\Phi_{t+h}^0)^{-1}\big(\Phi_t^0(x)+h b_t(\Phi_t^0(x))\big)-x}{h}\right).
    \end{align*}
   For the ease of notation, set $y:=\Phi_t^0(x)$. Then we need to show
    \[
    \lim_{h\to 0}\frac{(\Phi_{t+h}^0)^{-1}\big(y+h b_t(y)\big)-x}{h}=0, 
    \]
    or equivalently, using that $x=(\Phi_{t+h}^0)^{-1}(\Phi_{t+h}^0(x))$,  
    \begin{equation}\label{eq:goal}
    \lim_{h\to 0}\frac{(\Phi_{t+h}^0)^{-1}\big(y+h b_t(y)\big)-(\Phi_{t+h}^0)^{-1}(\Phi_{t+h}^0(x))}{h}=0.
    \end{equation}
    Observe that, by the semigroup-like identity \eqref{eq:semigroup-like} we deduce that for every $s,s' \in [0,1]$ it holds
    \[
    (\Phi_s^{s'})^{-1} = \Phi_{s'}^s.
    \]
    In particular, for $h >0$ sufficiently small, we have 
    \[
    (\Phi_{t+h}^0)^{-1} = \Phi_0^{t+h}
    \]
    which, combined with Lemma \ref{lemma:gronwall}(ii), conveys that also the map $(\Phi_{t+h}^0)^{-1}(\frarg)$ is $L$-Lipschitz, with $L:=\exp{\left(\int_0^1 \Lip(b_r)\, \dd r\right)}>0$. Therefore we have 
    \[
    |(\Phi_{t+h}^0)^{-1}\big(y+h b_t(y)\big)-(\Phi_{t+h}^0)^{-1}(\Phi_{t+h}^0(x))|\le L |y+h b_t(y)-\Phi_{t+h}^0(x)|.
    \]
    Thus, in order to show \eqref{eq:goal}, we just need to show that 
    \[
    \lim_{h\to0}\frac{|y+h b_t(y)-\Phi_{t+h}^0(x)|}{h}=0.
    \]
    Again by the semigroup law \eqref{eq:semigroup-like}, we can write $\Phi_{t+h}^0(x)=\Phi^t_{t+h}(\Phi^0_t(x))=\Phi^t_{t+h}(y)$. Therefore we have 
    \begin{align*}
    \frac{|y+h b_t(y)-\Phi_{t+h}^0(x)|}{h} & =
    \frac{|h b_t(y)+y-\Phi_{t+h}^t(y)|}{h}\\
    & \le \fint_t^{t+h}|b_t(y)-b_s(\Phi^t_s(y))|\,\dd s\\
    &\le \fint_t^{t+h}|b_t(y)-b_s(y)|\,\dd s+\fint_t^{t+h}|b_s(y)-b_s(\Phi^t_s(y))|\,\dd s\\
    &\le \fint_t^{t+h}|b_t(y)-b_s(y)|\,\dd s+\fint_t^{t+h} \Lip(b_s)|y-\Phi^t_s(y)|\,\dd s.
    \end{align*}
    We now invoke Lemma \ref{lemma:uniform_Leb_points}: we find a negligible set $N \subset \R$ such that for every $t \in N^c$ the first summand in the right hand side goes to zero. For the second summand, we write 
    \begin{align*}
    \fint_t^{t+h} \Lip(b_s)|y-\Phi^t_s(y)|\,\dd s & = \fint_t^{t+h} \Lip(b_s) | \Phi^t_t(y) -\Phi^t_s(y)|\,\dd s \\ 
    & \le \fint_t^{t+h} \Lip(b_s)\int_t^s |b_r(\Phi^t_r(y))|\, \dd r \,\dd s \\
    & \le \fint_t^{t+h} \Lip(b_s)\,\dd s \int_t^{t+h} \|b_r\|_{C^0(\R^d)} \, \dd r.
    \end{align*}
    The first factor is finite since $t \in N$ (again by Lemma \ref{lemma:uniform_Leb_points}) while the second converges to $0$ as $h \to 0$, since the function $r \mapsto \|b_r\|_{C^0(\R^d)}$ is integrable by assumption. This shows that \eqref{eq:goal} holds true, and concludes the proof.
\end{proof}

\section{The decomposability bundle for \texorpdfstring{$AC_t\Lip_x$}{AC Lip} functions}\label{sec:bundle}

In this section we are going to extend some results concerning the decomposability bundle, which was introduced in \cite{AM}. The goal is to pass from Lipschitz functions to functions that are Lipschitz in space, but only $AC$ in time. In the following we will consider only measures that admit a decomposition $\mu=\Lscr^1\otimes\mu_t$ (in some cases also with $|\mu_t|(\R^d)$ uniformly bounded), since these are the ones we are going to apply the theorem to. It seems that extending the following results to an arbitrary measure $\mu$ involves some subtleties that could be explored in a separate project.

\begin{definition}\label{def:AC_t_Lip_x}
We say that a function $f:\R\times\R^d\to \R$ belongs to the space $AC_t\Lip_x$ if there exist a constant $C>0$ and a function $g\in L^1(\R)$ such that
\[
    |f(t,x)-f(s,y)|\le C|x-y|+\int_s^t g(\tau)\,\dd\tau\qquad\text{for every $s,t\in\R$ and $x,y\in\R^d$.}
\]
\end{definition}
The smallest constant $C$ for which the inequality holds is denoted by $\Lip_x(f)$. Any function $g$ satisfying the inequality above will be referred to as an \emph{upper gradient} for $f$. Usually such a term is used in a slightly different way, but we use it to denote only the time component $g$.

\subsection{Differentiability along the bundle}
We start by proving that a function in $AC_t\Lip_x$ retains the same differentiability properties of Lipschitz functions, with respect to measures of the form $\mu=\Lscr^1\otimes \mu_t$.

\begin{theorem}\label{thm:differentiability}
    Let $\mu=\Lscr^1\otimes \mu_t$
    be a measure in $\R\times\R^d$, and let $f\in AC_t\Lip_x(\R\times\R^d)$. Then $f$ is differentiable along $V(\mu,(t,x))$ for $\mu$-a.e. $(t,x)$.
\end{theorem}

\begin{remark}
    It is not entirely clear to us how to extend the previous result to an arbitrary measure $\mu$. One might even wonder whether the differentiability along the bundle $V(\mu,\frarg)$ could be proved for functions that are $AC$ in all variables, namely such that
    \[
    |f(x_1,\ldots,x_n)-f(y_1,\ldots, y_n)|\le \sum_{i=1}^n \int_{[x_i,y_i]} g(t)\,\dd t
    \]
    for some $g\in L^1(\R)$. We leave this question for a future project.
\end{remark}

In order to show Theorem \ref{thm:differentiability}, we will need some preliminary lemmas. Given a closed set $E\subset \R$, we can write $E^c=\bigcup_{\ell\in\mathbb{N}} I_\ell$, where $I_\ell$ are at most countably many pairwise disjoint open intervals. If $f:\R\to\R$ is a continuous function, let us denote by $L(f,E) \colon \R \to \R$ the function that coincides with $f$ on $E$ and that interpolates linearly between the endpoint values on each $I_\ell$, namely
\begin{equation}\label{eq:def_L(f,E)}
    L(f,E)(t):=
    \begin{cases}
        f(t) & \text{if $t\in E$}\\
        \lambda f(r_\ell)+(1-\lambda)f(s_\ell) & \text{if $t=\lambda r_\ell+(1-\lambda)s_\ell\in I_\ell=(r_\ell,s_\ell)$}
    \end{cases}
\end{equation}

\begin{figure}
\centering
\begin{tikzpicture}[scale=.8]
  \begin{axis}[clip=false, axis lines = middle, xtick=\empty, ytick=\empty] 
 
        \addplot [domain=0:1, samples=100, blue] {x^2};

        \addplot [domain=1:2, samples=100, dashed, blue] {12*x^2 - 35*x + 24};

        \addplot [domain=2:2.5, samples=100, blue] {2-(x-2)^3};

        \addplot [domain=2.5:4, samples=100, dashed, blue] {
        -2/3*x^2 + 29/12*x};

        \addplot [domain=4:5, samples=100, blue] {-1+(x-4)^2};
        
        \addplot[thick, red] coordinates { (1,1) (2,2)};
    
        \addplot[thick, red] coordinates { (2.5, 1.875) (4,-1)};

        \addplot[thick] coordinates {(1, -0.1) (1, 0.1)};
        \addplot[thick] coordinates {(2, -0.1) (2, 0.1)};
        \addplot[thick] coordinates {(2.5, -0.1) (2.5, 0.1)};
        \addplot[thick] coordinates {(4, -0.1) (4, 0.1)};
  
    \end{axis}

        \node[] at (2,2) {$I_1$};
        \node[] at (4.5,2) {$I_2$};
\end{tikzpicture}
\caption{Visual representation of the function $L(f,E)$. On the open intervals $I_1$ and $I_2$ we replace the function $f$ (dashed graph) with the linear interpolation between endpoints.}\label{fig:interpolation}\end{figure}
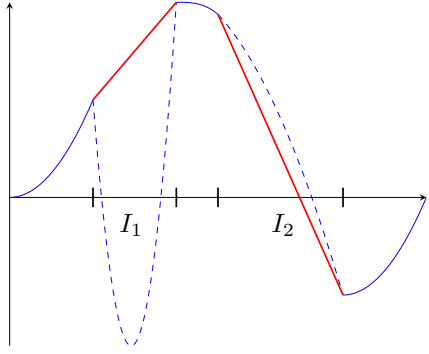

\begin{lemma}[Approximation of $AC$ functions]\label{lemma:one-dimensional}
    Let $g\in L^1(\R)$. Then there exists a sequence of closed sets $E_j\subset\R$, $j\in\mathbb{N}_+$, and a dimensional constant $C$ such that the following holds: if $f\in AC(\R)$ admits $g$ as upper gradient then:
    \begin{enumerate}
        \item $f|_{E_j}$ is $j$-Lipschitz;
        \item $\Lscr^1(E_j^c)\le C\|g\|_{L^1}/j$;
    \item Moreover, setting $f_j:=L(f,E_j)$, we additionally have:
    \begin{enumerate}
        \item $f_j\to f$ uniformly;
        \item $ f'_j(x)\to f'(x)$ for $\Lscr^1$-a.e. $x$;
        \item $f'_j\to f'$ in $L^1$.
    \end{enumerate}
    \end{enumerate}
\end{lemma}

\begin{proof}
    Recall from Section \ref{sec:maximal} the definition and properties of the maximal function $Mg$. For $j\in\mathbb{N}$ let $F_j:=\{x\in\R:\, Mg(x)\le j\}$. Then, for some dimensional constant $C$, $f|_{F_j}$ is $j$-Lipschitz and $\Lscr^1(F_j^c)< C\|g\|_{L^1}/j$. By inner regularity of the Lebesgue measure, for each $j$ we can select a closed subset $E_j\subseteq F_j$ satisfying the same properties, so that (1) and (2) are satisfied.
    The convergence stated in (a) follows from the definition of $f_j$ and the absolute continuity.
    Writing $E_j^c=\bigcup_\ell I_\ell$ we have that $f_j'(x)=f'(x)$ for $\Lscr^1$-a.e. $x\in E_j$ by locality of the derivative. By Point (2) we deduce (b). We are left to prove (c). For every $I_\ell=(r_\ell,s_\ell)$ we have
    \[
    \int_{I_\ell} |f_j'(x)|\,\dd x=\left|\int_{I_\ell} f_j'(x)\,\dd x\right|=|f(s_\ell)-f(r_\ell)|=\left|\int_{I_\ell} f'(x)\,\dd x\right|\le  \int_{I_\ell} |f'(x)|\,\dd x
    \]
    It follows that 
    \[
    \int_{E_j^c}|f_j'(x)|\,\dd x\le \int_{E_j^c} |f'(x)|\,\dd x
    \]
    which goes to zero since $f'\in L^1(\R)$ and, by Point (2), $\Lscr^1(E_j^c)\to 0$. This implies (c) and concludes the proof.
\end{proof}

Now we turn to a more general version of the previous lemma, that holds in any dimension.

\begin{lemma}[Approximation of $AC_t\Lip_x$ functions]\label{lemma:approximation_multi-dimensional}
    Let $f\in AC_t \Lip_x(\R\times\R^d)$. Then there exists a sequence of closed sets $E_j\subset \R$ and Lipschitz functions $f_j$ such that:
    \begin{enumerate}
        \item $f|_{E_j\times\R^d}$ is Lipschitz and $\Lip_x(f_j)\le \Lip_x(f)$;
        \item $f_j=f$ on $E_j\times\R^d$;
        \item $f_j\to f$ uniformly;
        \item $f-f_j$ is differentiable, with differential zero, at all points $(t,x)$ such that $t$ is a density point of $E_j$ (in particular, for $\Lscr^1$-a.e. $t\in E_j$).
    \end{enumerate}
\end{lemma}

\begin{proof}
    For every $x\in \R^d$, we write $f_x(t):=f(t,x)$. Applying Lemma \ref{lemma:one-dimensional}, and observing that $g$ is an upper gradient for $f_x$ for every $x$, we can find closed sets $E_j$ with $\Lscr^1(E_j^c)\le C\|g\|_{L^1}/j$ and with $f_x|_{E_j}$ $j$-Lipschitz. We then linearly extend this function separately on each line: recalling \eqref{eq:def_L(f,E)} we set
    \[
    f_j(t,x)=L(f_x,E_j)(t).
    \]
    We shall show that the sequence $f_j$ satisfies the required conditions (1)-(4).

    (1) For fixed $j$, $f_j(\cdot,x)$ is uniformly Lipschitz in $x\in \R^d$ by construction. Moreover $f_j(t,\cdot)$ is Lipschitz for every $t\in E_j$, because in this case $f_j(t,\cdot)=f(t,\cdot)$ and $f$ is Lipschitz in the space variable by assumption. Suppose instead that $t\in E_j^c$ and let $\ell \in \N$ be such that $t\in I_\ell=(r_\ell,s_\ell)$. Then for every $x,x'\in\R^d$
    \[
    |f(t,x)-f(t,x')|\le\max\{|f(r_\ell,x)-f(r_\ell,x')|,|f(s_\ell,x)-f(s_\ell,x')|\}\le \Lip_x(f)|x-x'|.
    \]
    It follows that $f_j$ is Lipschitz, and that $\Lip_x(f_j)\le \Lip_x(f)$.
    
    (2) This holds by construction. 

    (3) If $t\in E_j$ then $f_j(t,x)=f(t,x)$ for every $x$. If instead $t\in I_\ell=(r_\ell,s_\ell)\subseteq E_j^c$, then
    \[
    |f_j(t,x)-f(t,x)|\le 2\int_{I_\ell} g(s)\,\dd s
    \]
    which goes to zero as $j\to\infty$, because $|I_\ell|\le |E_j^c|\to 0$.

    (4) By construction of $E_j$, for all $t\in E_j$ we have $Mg(t)\le j$. From this and the definition of $AC_t\Lip_x$ it follows that
    \begin{equation}\label{eq:lip_estimate_maximal}
    |f(t,x)-f(t',x)|\le \int_{[t,t']} g(s)\,\dd s \le j|t-t'|.
    \end{equation}
    Since $f_j$ is $j$-Lipschitz, it follows that the following estimate holds for $f-f_j$:
    \begin{align}
    \begin{aligned}\label{eq:f-f_j}
    |(f-f_j)(t,x)-(f-f_j)(t',x)|&\le |f(t,x)-f(t',x)|+|f_j(t,x)-f_j(t',x)|\\
    & \le 2j|t-t'|.
    \end{aligned}
    \end{align}
    Moreover, the function $f-f_j$ is zero on $E_j\times\R^d$ by construction. From \eqref{eq:f-f_j} it then follows that
    \[
    |(f-f_j)(t,x)|\le 2j\dist(t,E_j).
    \]
    If $t$ is a density point of $E_j$ this yields that $f-f_j$ has differential zero at $(t,x)$.
\end{proof}

\begin{proof}[Proof of Theorem \ref{thm:differentiability}]
    We consider the functions $f_j$ built in Lemma \ref{lemma:approximation_multi-dimensional}. We write $f=(f-f_j)+f_j$ and observe that:
    \begin{itemize}
        \item $f_j$ is Lipschitz and thus differentiable in direction $V(\mu,(t,x))$ at $\mu$-a.e. $(t,x)$ by Theorem \ref{thm:bundle_AM};
        \item $(f-f_j)(t,x)$ has differential zero for $\Lscr^1$-a.e. $t\in E_j$; in particular $f(t,x)$ is also differentiable in direction $V(\mu,(t,x))$ for $\Lscr^1$-a.e. $t\in E_j$.
    \end{itemize}
     Since $\Lscr^1(E_j^c)\to0$, we deduce that $f(t,x)$ is differentiable in direction $V(\mu,(t,x))$ for $\Lscr^1$-a.e. $t\in \R$. In particular, the conclusion holds at $\Lscr^1\otimes\mu_t$-almost all points.   
     As the measure $\mu$ is of the form $\Lscr^1\otimes \mu_t$, it follows that $f$ is differentiable in direction $V(\mu,x)$ at $\mu$-a.e. $(t,x)$, concluding the proof.
\end{proof}

\subsection{Approximation with Lipschitz functions and pushforward formula}

Our next goal is to derive a pushforward formula for $AC_t\Lip_x$ functions. This task is based on an approximation result with Lipschitz functions, for which we know that the pushforward formula \eqref{eq:pushforward_lipschitz} holds. The reader might feel that the following proof is very similar to that of Lemma \ref{lemma:approximation_multi-dimensional}. Although this is partially true, we first needed to establish the differentiability of $f$ along the bundle stated in Theorem \ref{thm:differentiability}, so that the various terms below involving $D_Vf$ are well defined. 

Compare also the following lemma with \cite[Corollary~8.3]{AM}. Observe that the approximating sequence $f_j$ is universal, namely it works for every measure $\mu$ of the form considered below.

\begin{lemma}[Approximation with Lipschitz functions]\label{lemma:Lipschitz_approximation}
    Let $f\in AC_t \Lip_x(\R\times\R^d)$. Then there exists a sequence $f_j$ of Lipschitz functions such that the following holds.
    
    Let $T=\tau\mu$ be a normal current satisfying:
    \begin{enumerate}[(i)]
        \item $\mu=\Lscr^1\otimes\mu_t$ with
        \[
        m:=\sup_t |\mu_t|(\R^d)<\infty
        \]
        \item $\tau=(1,\tilde{b}_t)\wedge \tau_t$ for some unit horizontal $\tau_t$ and for some vector field $\tilde{b}$.
    \end{enumerate}
    Then, denoting by $V:=V(\mu,x)$ the decomposability bundle of $\mu$, we have:
    \begin{enumerate}
        \item $f_j\to f$ uniformly;
        \item $d_V f_j\to d_V f$ for $\mu$-a.e. $x$;
        \item $d_V f_j\to d_V f$ in $L^1(\mu)$;
        \item $\Wedge^k d_V f_j[\tau] \to \Wedge^k d_V f[\tau]$ in $L^1(\mu)$.
    \end{enumerate}
\end{lemma}

Observe that by Theorem \ref{thm:differentiability} the quantity $d_V f$ (and thus $\Wedge^k d_V f$) is well-defined at $\mu$-almost every point.
Observe also that the vector field $\tilde{b}$ automatically satisfies
\begin{equation}\label{eq:b_integrability}
\int_0^1 \int_{\R^d} |\tilde{b_t}|\,\dd\mu_t \,\dd t<\infty,
\end{equation}
because by assumption $T=\tau\mu$ has finite mass. In the proof of the uniqueness result (Theorem \ref{thm:uniqueness}) we will use this lemma with $\tilde{b}=b$, but we decided to keep a separate notation for this lemma and for the following proposition, since the only assumption on $\tilde b$ is its integrability given by \eqref{eq:b_integrability} (while we always assume that $b$ satisfies the more restrictive condition \eqref{eq:Lip_assumption}).

\begin{proof}
    We consider the sequence $f_j$ built in Lemma \ref{lemma:approximation_multi-dimensional}. Then Point (1) follows immediately. Point (2) also follows from the construction, since $d_V f_j=d_V f$ for $\mu$-a.e. point in $E_j\times\R^d$ by Lemma \ref{lemma:approximation_multi-dimensional}(4). Let us show (3) and (4).
    
    (3) Let us write $E_j^c=\bigcup_\ell I_\ell$, and let us fix a point $(t,x)$, with $t\in I_\ell=(r_\ell,s_\ell)$, where $f_j$ is differentiable along $V(\mu,x)$. We claim that
    \[
    |d_V f_j(t,x)|\le \frac{|f(s_\ell,x)-f(r_\ell,x)|}{|s_\ell-r_\ell|}+\Lip_x(f).
    \]
    Indeed, given any unit vector $(e_t,e_x)\in\R\times\R^d$, we have the following estimate for the directional derivative of $f$ in direction $(e_t,e_x)$:
    \begin{align}
    \begin{aligned}\label{eq:incremental_quotient}
        \frac{1}{h}\big(f_j((t,x)+h(e_t,e_x))-f_j(t,x)\big)& =\frac{1}{h}\big(f_j((t,x)+h(e_t,e_x))-f_j((t,x)+h(e_t,0)\big)\\
        &\qquad+ \frac{1}{h}\big(f_j((t,x)+h(e_t,0))-f_j(t,x)\big)\\
        &\le \frac{1}{h}\Lip_x(f_j)|h e_x|+|e_t|\frac{|f(s_\ell)-f(r_\ell)|}{|s_\ell-r_\ell|}\\
        &\le  \Lip_x(f)|e_x|+|e_t|\frac{|f(s_\ell)-f(r_\ell)|}{|s_\ell-r_\ell|}.
    \end{aligned}
    \end{align}
    It immediately follows that 
    \[
    |d_Vf_j(t,x)|\le \Lip_x(f)+\frac{|f(s_\ell)-f(r_\ell)|}{|s_\ell-r_\ell|},
    \]
    and thus also that
    \begin{align*}
    \int_{E_j^c\times\R^d}|d_V f_j| \dd\mu& =\int_{E_j^c} \int_{\R^d}|d_V f_j| \dd\mu_t\,\dd t\\
    &\le m\sum_\ell \int_{I_\ell} \left(\frac{|f(s_\ell)-f(r_\ell)|}{|s_\ell-r_\ell|}+\Lip_x(f)\right)\,\dd t\\
    &\le m \sum_\ell |f(s_\ell)-f(r_\ell)|+m|E_j^c|\Lip_x(f)\\
    &\le m\int_{E_j^c} g(r)\,\dd r +m|E_j^c|\Lip_x(f).
    \end{align*}
    The quantity in the last line goes to zero because $|E_j^c|\to 0$. This proves (3).

    (4) 
    From the second to last line in \eqref{eq:incremental_quotient}, and recalling that by Lemma \ref{lemma:approximation_multi-dimensional}(1) $\Lip_x(f_j)\le \Lip_x(f)$, one obtains that for $t\in I_\ell\subseteq  E_j^c$
    \[
    df_j [(1,\tilde{b}_t)]\le \Lip_x(f)|\tilde{b}_t|+\frac{|f(s_\ell)-f(r_\ell)|}{|s_\ell-r_\ell|}
    \]
    It follows that
    \[
    |\Wedge^k d_V f_j[\tau]|\le \Lip_x(f)^{k-1}\left( \Lip_x(f)|\tilde{b}_t|+\frac{|f(s_\ell)-f(r_\ell)|}{|s_\ell-r_\ell|}\right).
    \]
    For the second term, one can argue as in Point (3). Instead the first term reduces to
    \[
    \int_{E_j^c\times\R^d} \Lip_x(f)^k |\tilde{b}_t| \,\dd \mu=\Lip_x(f)^k \int_{E_j^c} \int_{\R^d} |\tilde{b}_t| \,\dd\mu_t \,\dd t
    \]
    which again goes to zero as $j\to\infty$. To see this, we just observe that the map
    \[
    t\mapsto\int_{\R^d} |\tilde{b}_t|\,\dd\mu_t
    \]
    is integrable on $[0,1]$ since
    \[
    \int_0^1\int_{\R^d} |\tilde{b}_t|\,\dd\mu_t \,\dd t \le \int_0^1 \int_{\R^d} |(1,\tilde{b}_t)|\,\dd\mu_t\,\dd t=\int_{[0,1]\times\R^d} |(1,\tilde{b})|\,\dd\mu=\Mbf(T)<\infty.\qedhere
    \]
\end{proof}

\begin{proposition}[Pushforward formula for $AC_t\Lip_x$ maps]\label{prop:pushforward}
    Let $f\in AC_t \Lip_x(\R\times\R^d)$ be a proper map. Let $T=\tau\mu$ be a normal current satisfying:
    \begin{enumerate}[(i)]
        \item $\mu=\Lscr^1\otimes\mu_t$ with
        \begin{equation}\label{eq:assumption_bounded_mass}
        m:=\sup_t |\mu_t|(\R^d)<\infty;
        \end{equation}
        \item $\tau=(1,\tilde{b}_t)\wedge \tau_t$ for some unit horizontal $\tau_t$ and for some vector field $\tilde{b}$.
    \end{enumerate}    
    Let $f_j$ be a sequence of Lipschitz functions satisfying the conditions (1)-(4) in Lemma \ref{lemma:Lipschitz_approximation}. Then:
    \begin{enumerate}
        \item The limit
        \[
        f_* T:=\lim_{j\to\infty} (f_j)_* T
        \]
        exists in the strong topology, and defines a finite mass current. Moreover, for any other sequence $f_j$ satisfying (1)-(4) of Lemma \ref{lemma:Lipschitz_approximation} the limit is the same.
        \item $f_*T$ satisfies the pushforward formula
        \[
        \langle f_*T,\omega\rangle=\int \langle \omega(f(z)),\Wedge^k d_T f[\tau(z)] \rangle\,\dd\mu(z)\qquad\text{for every $\omega\in \Dscr^k(\R^d)$.}
        \]
        \item If $\partial T=\tau'\mu'$, with $\mu'=\Lscr^1\otimes \mu_t'$ and $\tau'=(1,\tilde {b}_t')\wedge\tau_t'$ satisfies $(i)$ and $(ii)$ 
        then $f_*T$ is normal and $\partial (f_*T)=f_*(\partial T)$.
    \end{enumerate} 
\end{proposition}

\begin{proof}
    Given that $f_j$ are Lipschitz we can apply the pushforward formula \eqref{eq:pushforward_lipschitz} to obtain
    \begin{align}\label{eq:pushforward_f_j}
    \langle  (f_j)_*T,\omega\rangle= \int \langle \omega(f_j(z)),\Wedge^k d_T f_j [\tau(z)] \rangle\,\dd\mu(z)\qquad\text{for every $\omega\in \Dscr^k(\R^d)$.}
    \end{align}
    By Condition (4) in Lemma \ref{lemma:Lipschitz_approximation} we know that $\Wedge^k d_Tf_j[\tau]\to \Wedge^kd_Tf[\tau]$ in $L^1(\mu)$. Moreover $f_j\to f$ uniformly, and thus also $\omega \circ  f_j\to \omega\circ f$ uniformly. From this it follows that we can pass to the limit in \eqref{eq:pushforward_f_j} and obtain
    \[
    \lim_{j\to\infty}\langle (f_j)_*T,\omega\rangle=\int\langle \omega(f(z)),\Wedge^k d_T f[\tau(z)]\rangle\,\dd\mu(z)\qquad\text{for every $\omega\in \Dscr^k(\R^d)$.}
    \]
    This proves that the limit is well-defined, at least in the weak-* topology, does not depend on the chosen sequence $f_j$, and moreover the pushforward formula holds. Since $\Wedge^kd_Tf[\tau]\in L^1(\mu)$, it follows that $f_*T$ has finite mass. Moreover, this also yields that the convergence happens in the strong topology. This proves (1) and (2). To prove (3) we apply the pushforward formula to $\partial T$ and use that, for Lipschitz maps, the pushforward commutes with the boundary to obtain
    \begin{align*}
    \langle \partial(f_*T),\omega\rangle&=\langle f_*T, d\omega\rangle\\
    &=\lim_{j\to\infty}\langle (f_j)_*T,d\omega\rangle\\
    &=\lim_{j\to\infty}\langle (f_j)_*\partial T,\omega\rangle\\
    &= \langle f_*\partial T,\omega\rangle.
    \end{align*}
    Thanks to (i) applied to $\partial T$ we know that this defines a current with finite mass, and thus $f_*T$ is normal.
\end{proof}

\section{Existence}\label{sec:existence}

In this section we establish the following existence result:

\begin{theorem}[Existence]\label{thm:existence}
Let $b:[0,1]\times\R^d\to\R^d$ be a vector field satisfying \eqref{eq:Lip_assumption} and let $(\Phi^0_t)_t$ be its flow. Moreover, let $\bar T\in \Nrm_k(\R^d)$. Then, the family of currents $T_t=(\Phi^0_t)_* \bar T$ is a solution to
\[
\begin{cases}
    \frac{\dd}{\dd t} T_t+\Lcal_{b_t}T_t=0\\
    T_0=\bar T.
\end{cases}
\]
\end{theorem}

\begin{proof} We subdivide the proof in some steps.

    \textit{Step 1.} Write $\bar T = \bar \tau \bar \mu$ and $\partial \bar T = \bar \sigma\bar \nu$. Define 
    \[
    C = \bbb{ 0,1 } \times \bar T = [\ee_0 \wedge (i_t)_*\bar \tau] \L^1\otimes \bar \mu = [\ee_0 \wedge (i_t)_*\bar \tau] \L^1\times \bar \mu
    \]
    and observe that $\partial C \restrict ((0,1) \times \R^d)  = - \bbb{0,1} \times \partial \bar T$. 
    Define 
    \[
    Z := \Psi_{*} C, 
    \]
    where $\Psi$ is the map defined in \eqref{eq:def_Psi}. 
    We now apply the formula for the pushforward in Proposition \ref{prop:pushforward} to the current $C$. Notice that the assumptions (i) and (ii) on the uniform mass of the measures are automatically satisfied, as $\bar \mu$ is independent of $t$ and $\tilde b=0$. Therefore we get that the current $Z$ can be written as $Z = z\lambda$, where 
    \[
    \lambda = \Psi_{\#}(\L^1\otimes\mu) = \L^1 \otimes {(\Phi^0_t)}_{\#} \mu =: \L^1 \otimes \mu_t 
    \]
    and 
    \begin{align*}
    z(s,y) & = \Psi_*([\ee_0 \wedge (i_s)_*\bar \tau])(s,y)\\ 
    & = D\Psi(\Psi^{-1}(s,y)) [(\ee_0 \wedge (i_s)_*\bar \tau)(\Psi^{-1}(s,y))]\\ 
    & = D\Psi(\Psi^{-1}(s,y)) [\ee_0] \wedge D\Psi(\Psi^{-1}(s,y))[((i_s)_*\bar \tau)(\Psi^{-1}(s,y))] \\
    & = (1,b_s(y)) \wedge D\Psi(\Psi^{-1}(s,y))[((i_s)_*\bar \tau)(\Psi^{-1}(s,y))] \\
    & = (1,b_s(y)) \wedge [\Psi_* ((i_s)_* \bar\tau)](s,y)\\
    & = (1,b_s(y)) \wedge [(\Psi_* \circ (i_s)_* )\bar\tau](s,y)\\
    & = (1,b_s(y)) \wedge [(\Psi \circ i_s)_* \bar\tau](s,y)\\
    & = (1,b_s(y)) \wedge [(i_s)_* (\Phi^0_s)_* \bar \tau] (s,y). 
    \end{align*}
    Here we used the fact that $D\Psi(\Psi^{-1}(s,y)) [\ee_0]= (1,b_s(y))$ by Lemma \ref{lemma:diff_flow}, Point (1), and from the second to last to the last line we used Lemma \ref{lemma:chain_rule}, together with the identity $\Psi \circ i_s = i_s \circ \Phi_s^0$, to exchange the order of the pushforwards.

    \textit{Step 2.} Now we test $\partial Z$ against forms of the kind $\eta = \tbf^* \alpha \wedge \pbf^* \beta$. 
    We have $d\eta = \tbf^*d\alpha \wedge \pbf^*\beta + \tbf^*\alpha \wedge \pbf^*d\beta$. So by Lemma \ref{lemma:wedge_perp}(i) we obtain  
    \begin{align*}
        \langle Z, \tbf^*d\alpha \wedge \pbf^*\beta\rangle & = \int_0^1 \int_{\R^d} \langle (1,b_s) \wedge [(i_s)_* (\Phi^0_s)_* \bar \tau] ,  \tbf^*d\alpha \wedge \pbf^*\beta \rangle \, \dd \mu_s \, \dd s \\ 
        & =  \int_0^1 \int_{\R^d} \langle (1,b_s),  \tbf^*d\alpha \rangle \, \langle (i_s)_* (\Phi^0_s)_* \bar \tau ,\pbf^*\beta\rangle \,\dd \mu_s \, \dd s \\ 
        & =  \int_0^1 \int_{\R^d} \alpha'(s) \, \langle \tau_s,\beta\rangle \,\dd \mu_s \, \dd s \\ 
        & = \int_0^1 \alpha'(s) \langle T_s,  \beta\rangle \, \dd s 
    \end{align*}
    where we have defined $\tau_s:=(\Phi^0_s)_* \bar \tau$, so that by definition $T_s= \tau_s \mu_s$.
    
    For the second summand, invoking  Lemma \ref{lemma:wedge_perp}(ii) we have 
    \begin{align*}
        \langle Z, \tbf^*\alpha \wedge \pbf^*d\beta \rangle & = \int_0^1 \int_{\R^d} \langle (1,b_s) \wedge [(i_s)_* (\Phi^0_s)_* \bar \tau] ,  \tbf^*\alpha \wedge \pbf^*d\beta \rangle \, \dd \mu_s \, \dd s \\ 
        &= \int_0^1 \int_{\R^d} \alpha(s) \langle b_s \wedge [(\Phi^0_s)_* \bar \tau], d\beta \rangle \dd \mu_s \, \dd s \\ 
        & =\int_0^1 \alpha(s)\langle \partial(b_s \wedge T_s), \beta \rangle \, \dd s.  
    \end{align*}

    \textit{Step 3.} By Proposition \ref{prop:pushforward}(3), we have 
    \[
    \partial Z \restrict ((0,1) \times \R^d)   = \Psi_{*} (\partial C)\restrict ((0,1) \times \R^d) = \Psi_{*} (
    \partial C\restrict ((0,1) \times \R^d) )= \Psi_{*} (- \bbb{0,1} \times \partial \bar T).
    \]
    Now the orienting field is 
    \[
    \xi(s,y) = \Psi_{*}(-\ee_0 \wedge \overline\sigma)(s,y) = - (1,b_s(y)) \wedge \Psi_{*}((i_{s})_{*} \overline\sigma)(s,y)
    \]
    where we have used again Lemma \ref{lemma:diff_flow} Point (1) to deduce   
    \[
    \Psi_{*}(\ee_0)(s,y) = D\Psi(\Psi^{-1}(s,y))[1,0] = (1,b_s(y)). 
    \]
    Notice also that the mass measure of $\partial Z$ is 
    \[
    \|\partial Z\|=\Psi_{\#}(\|\partial C\|) = \Psi_{\#} (\L^1 \otimes \overline{\nu} ) = \L^1 \otimes \nu_s
    \]
    where we have set $\nu_s := (\Phi^0_s)_{\#}\overline{\nu}$.
    Testing against forms of the type $\eta=\tbf^*\alpha\wedge \pbf^*\beta$, with $\alpha\in \Crm^\infty_c((0,1))$ and $\beta\in \Dscr^k(\R^d)$, we therefore have
    \begin{align*}
        \langle \partial Z, \eta \rangle & = \int_0^1 \int_ {\R^d} \langle \xi, \eta \rangle \, \dd \nu_s \dd s \\ 
        & = - \int_0^1 \int_{\R^d} \langle (1, b_s) \wedge \Psi_{*}((i_{s})_{*} \overline\sigma) , \eta \rangle \,\dd\nu_s \dd s \\ 
        & = - \int_0^1 \int_{\R^d} \langle (1, b_s) \wedge \Psi_{*}((i_{s})_{*}\overline\sigma) , \tbf^* \alpha \wedge \pbf^* \beta \rangle \,\dd\nu_s \dd s \\ 
        & = - \int_0^1 \alpha(s) \int_{\R^d} \langle (1, b_s) \wedge (i_s)_* (\Phi^0_s)_* \overline\sigma , \pbf^* \beta \rangle \,\dd\nu_s \dd s \\ 
        & = - \int_0^1 \alpha(s) \int_{\R^d} \langle b_s \wedge (\Phi^0_s)_* \overline\sigma ,  \beta \rangle \,\dd\nu_s \dd s \\ 
        & = - \int_0^1 \alpha(s) \langle b_s \wedge \partial T_s , \beta \rangle \dd s 
    \end{align*}
    where in the second to last line we have used Lemma \ref{lemma:wedge_perp}(ii) and the fact that 
    \[
    \partial T_s = (\Phi^0_s)_{*}(\partial \overline{T}) = [(\Phi^0_s)_{*}\overline{\sigma}] \, (\Phi^0_s)_{\#} (\overline{\nu})
    \]
    for every $s$.

    \emph{Step 4.}
Combining Step 2 and 3 we obtain that
		\begin{align*}
		-\int_0^1\alpha(s) \, \dprb{ b\wedge\partial T_s, \beta}\;\dd s 
		&=\dprb{ \partial Z,\eta} \\
		&=\dprb{  Z,d\eta} \\
		&= \int_0^1 \alpha'(s) \, \dprb{ T_s,\beta} + \alpha(s) \, \dprb{ b\wedge T_s,d\beta}\;\dd s
		\end{align*}
		for every $\alpha\in\Dscr^0((0,1))=\Crm_c^\infty((0,1))$ and every $\beta\in\Dscr^k(\R^d)$.
		Rearranging terms gives exactly the weak formulation of~\eqref{eq:GTE}.
\end{proof}

\begin{remark} Theorem \ref{thm:existence} could also be proved with an approximation argument that does not rely on the decomposability bundle.
Consider two functions $\phi \in\Crm_c^\infty((-1,1))$ and $\psi \in \Crm_c^\infty(B_1)$ and construct the associated mollifiers 
\[
\phi^\eps(t):=\frac{1}{\eps}\phi(t/\eps),\qquad \psi^\eps(x):=\frac{1}{\eps^d}\psi(x/\eps).
\]
Let $b^\eps:=b* (\phi^\eps\psi^\eps)$ (after extending $b$ to the identically zero vector field for $t<0$ or $t>1$) and let $\Phi^\eps_t$ be the corresponding flow (starting from time $t=0$). Then 
\begin{align*}
|\Phi_t(x) - \Phi^\eps_t(x) | & \le \int_0^t|b_s(\Phi_s(x)) - b_s^\eps(\Phi_s^\eps(x))|\, \dd s \\ 
& \le \int_0^t|b_s(\Phi_s(x)) - b_s(\Phi^\eps_s(x))| \,  \dd s + \int_0^t |b_s(\Phi^\eps_s(x))- b_s^\eps(\Phi^\eps_s(x))|\, \dd s \\ 
& \le \int_0^t \Lip_x(b_s)|\Phi_s(x)-\Phi^\eps_s(x)| \, \dd s + \int_0^t \|b_s-b_s^\eps\|_{\infty}\, \dd s
\end{align*}
and therefore by Gronwall we obtain 
\[
\| \Phi_t - \Phi^\eps_t \|_{\infty} \le \int_0^t \|b_s-b_s^\eps\|_{\infty}\, \dd s \cdot \exp{\left(\int_0^t \Lip_x(b_s)\, \dd s\right)}.
\]
Additionally, it can be shown that the Lipschitz constant in space of the maps $\Phi^\eps_t(\cdot)$ is equibounded. Indeed, a simple convolution argument yields 
\[
\int_0^1 \Lip(b^{\eps}_t) \, \dd t\le \int_0^1 \Lip(b_{t}) * \alpha^\eps(t) \, \dd t \le L,  
\]
whence $\Lip(\Phi_t^\eps) \le \exp(Lt)$. 
In order to show existence of solutions to \eqref{eq:GTE} it is therefore sufficient to consider the family of pushforwards $T^\e_t:=(\Phi^\e_t)_{*} \overline{T}$, which solves the equation for $b^\eps$ by \cite[Theorem 3.6]{BDR}. Their mass is equibounded (in view of the equi-Lipschitz bound on $\Phi^\eps_t(\cdot)$), and any limit point satisfies \eqref{eq:GTE} by linearity of the equation. Moreover, using that $\Phi^\eps_t$ converges uniformly to $\Phi^0_t$ with equi-Lipschitz constant, it can also be shown that the limit current $T_t$ coincides with $(\Phi^0_t)_{*} \overline{T}$ mimicking the argument in \cite[Lemma 7.4.3]{KrantzParks08book}.
\end{remark}

\section{Uniqueness}\label{sec:uniqueness}

\begin{theorem}[Uniqueness]\label{thm:uniqueness}
Let $b:[0,1]\times\R^d\to\R^d$ be a vector field satisfying \eqref{eq:Lip_assumption} and let $(\Phi_t^0)_t$ be its flow. Moreover, let $\bar T\in \Nrm_k(\R^d)$. Then, there is at most one solution to
\[
\begin{cases}
    \frac{\dd}{\dd t} T_t+\Lcal_{b_t}T_t=0\\
    T_0=\bar T
\end{cases}
\]
in the class $L^\infty([0,1]; \Nrm_k(\R^d))$. Moreover, the solution is given by $T_t=(\Phi^0_t)_*\bar T$.
\end{theorem}

\begin{proof}
    We first show the conclusion under the assumption that $\partial T_t=0$ for every $t \in [0,1]$.  

    \textit{Step 1.} 
    Let $(T_t)_{t \in (0,1)} \subset \Nrm_k(\R^d)$ with $\partial T_t=0$ ($t \in[0,1)$) be a weakly$^*$-continuous solution to~\eqref{eq:GTE} (see~\cite[Lemma 3.5(i)]{BDR} for why we may assume weak*-continuity). We decompose $T_t=\vec{T}_t\|T_t\|$, with $\vec{T}_t$ unit $k$-vectors.
	Let $\vec{T}:(0,1)\times\R^d\to \Wedge_k(\R\times\R^d)$ be the $k$-vector field defined $\L^1_t \otimes\|T_t\|$-almost everywhere by
	\[
	  \vec{T}(t,x):=(\iota_t)_* \vec{T}_t(x),
	 \]
	 where we recall that $\iota_t(x):=(t,x)$.
    Define now the space time currents $U=u\mu$ with $\mu=\Lscr^1\otimes\|T_t\|$ and $u(t,x)=(1,b(t,x))\wedge \vec{T}(t,x)$ and 
    \[
    W:=(\Psi^{-1})_*U.
    \]
    The map $\Psi^{-1}$ is still $AC_t\Lip_x$ and the map $t \mapsto \|T_t\|(\R^d)$ is essentially bounded by assumption, so we can define the pushforward through Proposition \ref{prop:pushforward}. In particular, we can write $W=w \nu$, where 
    \[
    \nu = \Psi^{-1}_{\#}(\L^1 \otimes \|T_t\|) = \L^1 \otimes [(\Phi_t^0)^{-1}_{\#} \|T_t\|] =: \L^1 \otimes \nu_t 
    \]
    and 
    \begin{align*}
    w(t,x)&=D\Psi^{-1}(\Psi(t,x))[(1,b(t,\Phi^0_t(x)))]\wedge D\Psi^{-1}(\Psi(t,x))[\vec T(t,\Phi^0_t(x))]\\
    &=(1,0)\wedge D\Psi^{-1}(\Psi(t,x))[\vec T(t,\Phi^0_t(x)]\\
    &=:(1,0)\wedge \tau(t,x).
    \end{align*}
    where we have used Lemma \ref{lemma:diff_flow} Point (2). 
    
    \emph{Step 2.} 
    We can estimate the mass of $U$ by 
    \[
    \Mbf(U) \le \int_{0}^1 \|b_t\|_{C^0(\R^d)}\Mbf(T_t) \, \dd t \le L \esssup_{t \in (0,1)} \Mbf(T_t) < \infty. 
    \]
    Moreover, applying \cite[Lemma 4.1]{BDR-Lipschitz} (replacing $b(x)$ with $b_t(x)$ everywhere) we have 
    \[
    \partial U \restrict (0,1) \times \R^d = 0
    \]
    and therefore by Proposition \ref{prop:pushforward}(Point 3), $\partial W \restrict (0,1) \times \R^d= 0$. This means that for every $\alpha\in\Dscr^0((0,1))=\Crm_c^\infty((0,1))$ and every $\beta\in\Dscr^k(\R^d)$, it holds
    \[
    0= \langle \partial W, \tbf^*\alpha \wedge \pbf^*\beta \rangle = \langle W, \tbf^*d \alpha \wedge \pbf^*\beta \rangle 
    \]
    where in the last equality we have used Lemma~\ref{lemma:wedge_perp}~(ii) (notice that $\mathrm{span}\,(\pbf^*d\beta)\perp (1,0)$ and $w(t,x) =(1,0)\wedge \tau(t,x)$). 
    By Lemma~\ref{lemma:wedge_perp}~(i) we have 
    \begin{align*}
    \langle w(t,x), (\tbf^*d\alpha \wedge \pbf^*\beta) (t,x) \rangle  & = \langle (1,0), \tbf^*d\alpha(t,x) \rangle \langle \tau(t,x) , \pbf^*\beta (t,x)\rangle  \\ 
    & = \alpha'(t)  \langle \tau(t,x) , \pbf^*\beta (t,x)\rangle  \\ 
    & = \alpha'(t)  \langle D\Psi^{-1}(\Psi(t,x))[\vec T(t,\Phi^0_t(x)], \pbf^*\beta (t,x)\rangle  \\ 
    &= \alpha'(t)  \langle [(\Psi^{-1})_{*}\vec T](t,x), \pbf^*\beta (t,x)\rangle.
    \end{align*}
    Now we observe that for every fixed $t \in (0,1)$ it holds 
    \[
    \pbf \circ \Psi^{-1} \restrict \{t\}\times \R^d = [(\Phi^0_{t})^{-1} \circ \pbf] \restrict \{t\}\times \R^d
    \]
    and therefore 
    \[
    (\pbf \circ \Psi^{-1})_{*}v = ((\Phi^0_{t})^{-1} \circ \pbf)_{*} v
    \]
    for any $k$-vector $v\in \Wedge_k(\{0\} \times \R^d)$. In particular, choosing $v= \vec{T}(\Psi(t,x))=T(t,\Phi^0_t(x))$ we get by Lemma \ref{lemma:chain_rule} Point (i) 
    \[
    (\pbf \circ \Psi^{-1})_{*}  \vec{T}(\Psi(t,x)) = ((\Phi^0_{t})^{-1} \circ \pbf)_{*} \vec{T}(\Psi(t,x)) = [(\Phi^0_t)^{-1}]_{*}[T_t(\Phi_t^0(x))]. 
    \]
    Therefore 
    \begin{align*}
        0= \langle W, \tbf^*d\alpha \wedge \pbf^*\beta \rangle & = \int_0^1 \int_{\R^d} \langle w, \tbf^*d\alpha \wedge \pbf^*\beta\rangle \, \dd \nu_t \, \dd t \\ 
        & =  \int_0^1 \int_{\R^d}\alpha'(t) 
        \langle (\Phi^0_t)^{-1}_{*}[T_t(\Phi_t^0(x))],\beta(x) \rangle \, \dd [(\Phi^0_t)^{-1}_{\#}\|T_t\|](x) \, \dd t \\ 
        & = \int_0^1 \alpha'(t) \langle (\Phi_t^0)^{-1}_{*} T_t, \beta \rangle \, \dd t.
    \end{align*}
    
    \emph{Step 3.} By Step 1 and Step 2 we have shown that 
		\[
		\int_0^1 \alpha'(t) \, \dprb{ (\Phi^0_{t})^{-1}_* T_t,\beta} \;\dd t=0\qquad\text{for every $\alpha\in\Dscr^0((0,1))$ and $\beta\in\Dscr^k(\R^d)$.}
		\]
		This implies that, for any fixed $\beta\in\Dscr^k(\R^d)$, the map $t\mapsto \dprb{ (\Phi^0_{t})^{-1}_* T_t,\beta} $ is constant, and in particular equal to its value at 0 (in this context recall that we have chosen the weakly$^*$-continuous representative of $t\mapsto T_t$). We conclude that $(\Phi^0_{t})^{-1}_* T_t = \overline T$ for every $t \in [0,1]$. Since for fixed $t \in [0,1]$ the map $\Phi^0_t$ is Lipschitz with inverse $(\Phi^0_{t})^{-1}$ we can write $T_t=(\Phi^0_t)_* \overline T$ which is the desired conclusion in the case $\partial T_t = 0$.

        \textit{Step 4.} Assume now that $(T_t)_{t \in (0,1)}$ is a solution of~\eqref{eq:GTE} with $T_t$ normal but not necessarily boundaryless. Testing with exact forms $\omega=d\eta$, we see that the boundaries $(\partial T_t)_{t\in (0,1)}$ solve the geometric transport equation with the same driving vector field $b$ and initial datum $\partial \overline{T}$. Since the $\partial T_t$ are boundaryless, by Step 3 we obtain that 
		\[
		\partial T_t=(\Phi^0_t)_* (\partial\overline{T})=\partial ((\Phi^0_t)_*\overline{T}).
		\]
		Consider now $S_t:=T_t-(\Phi^0_t)_* \overline{T}$. The $S_t$ are normal currents with equibounded mass w.r.t $t$ and $\partial S_t=0$. By the linearity of~\eqref{eq:GTE}, they still solve the same equation. Since $S_0=0$, again by Step 3 we conclude that $S_t=0$ for every $t \in (0,1)$, hence $T_t=(\Phi^0_t)_* \overline{T}$ for every $t \in (0,1)$.
	\end{proof}

\section{On non-uniqueness for finite mass currents}\label{sec:finite_mass}

One may wonder what happens if we consider \eqref{eq:GTE} in the class of currents with finite mass, but whose boundary might have unbounded mass. In this case the weak formulation of the PDE makes sense if $b$ is smooth, but if $b$ is only Lipschitz it is not clear how to even state the weak formulation. The issue lies in the term $b\wedge \partial T_t$ whose definition is unclear -- somewhat similarly to what happens when trying to define the product of a distribution with a Lipschitz function. 

In this section we shall see that even if one comes up with a notion of solution for finite mass currents, any sensible notion that satisfies very natural assumptions leads to non-uniqueness.

Let us consider the family $\Ccal=L^1([0,1];\Mrm_k(\R^d))$ consisting of all possible curves of currents $t \mapsto T_t$ that satisfy
\[
\int_0^1 \Mbf(T_t) \, \dd t < +\infty.
\]
If $b=b(t,x)$ is smooth, then the weak formulation \eqref{eq:def_weak_sense} is actually equivalent to the following one, which requires only the masses (and not the normal masses) to be integrable in time, namely $(T_t)_t \in \Ccal$ is a weak solution to \eqref{eq:GTE} if 
\begin{equation}\label{eq:def_weak_sol_smooth_field}
\int_0^1 \langle T_t, \omega\rangle \psi'(t) + \langle T_t,\Lcal_{b_t} \omega \rangle \psi(t) \,\dd t  = 0 
\end{equation}
for every $\omega \in \Dscr^k(\R^d)$ and every $\psi \in \Crm_c^\infty((0,1))$. Here $\Lcal_{b_t} \omega$ stands for the classical Lie derivative of the differential form $\omega$ as defined in Differential Geometry.
For a proof of this equivalence see \cite[Lemma 3.3]{BDR}.

\begin{definition}\label{def:natural_family}
Fix a Borel vector field $b$. We say that a family $\Scal_b\subset \Ccal$ is a \textit{natural family of solutions} to \eqref{eq:GTE} with vector field $b$ if the following conditions hold:
\begin{enumerate}
    \item[(H1)] \textit{Consistency}: If $b$ is smooth in some open set $\Omega$ and $(T_t)_{t\in[0,1]} \in \Ccal$ is a weak solution in the sense of \eqref{eq:def_weak_sol_smooth_field} supported in $\Omega$, then $(T_t)_{t\in[0,1]}\in\Scal_b$.
    \item[(H2)] \textit{Stability}: if $(T_t^j)_{t\in[0,1]}$, $j\in\mathbb{N}$, is a sequence in $\Scal_b$ with equibounded $L^1$-norm, namely 
    \[
    \sup_j \int_0^1 \Mbf(T_t^j) \, \dd t < \infty 
    \]
    and if $T_t=\lim_{j\to\infty} T_t^j$ for every $t$, then also $(T_t)_{t\in[0,1]}$ belongs to $\Scal_b$.
\end{enumerate}
\end{definition}

Observe that both conditions are very natural: the stability is natural because the PDE is linear, while the well-posedness for smooth vector fields is ensured by the following result:

\begin{proposition}
    Let $b=b(t,x)$ be a smooth globally bounded vector field. Then the problem
    \[
\begin{cases}
    \frac{\dd}{\dd t} T_t+\Lcal_{b_t}T_t=0\\
    T_0=\bar T
\end{cases}
\]
    admits a solution which is unique in the class $L^1([0,1];\Mrm_k(\R^d))$.
\end{proposition}

\begin{proof}[Sketch of the proof] 
One can repeat the argument in the proof of \cite[Theorem 3.6]{BDR} with minor changes, referring also to \cite[Section 8.1]{AGS} for the modifications needed to deal with the two-parameter flow.
\end{proof}

\begin{figure}
    \centering
    \resizebox{.4\linewidth}{!}{
\begingroup%
  \makeatletter%
  \providecommand\color[2][]{%
    \errmessage{(Inkscape) Color is used for the text in Inkscape, but the package 'color.sty' is not loaded}%
    \renewcommand\color[2][]{}%
  }%
  \providecommand\transparent[1]{%
    \errmessage{(Inkscape) Transparency is used (non-zero) for the text in Inkscape, but the package 'transparent.sty' is not loaded}%
    \renewcommand\transparent[1]{}%
  }%
  \providecommand\rotatebox[2]{#2}%
  \newcommand*\fsize{\dimexpr\f@size pt\relax}%
  \newcommand*\lineheight[1]{\fontsize{\fsize}{#1\fsize}\selectfont}%
  \ifx\svgwidth\undefined%
    \setlength{\unitlength}{148.13623143bp}%
    \ifx\svgscale\undefined%
      \relax%
    \else%
      \setlength{\unitlength}{\unitlength * \real{\svgscale}}%
    \fi%
  \else%
    \setlength{\unitlength}{\svgwidth}%
  \fi%
  \global\let\svgwidth\undefined%
  \global\let\svgscale\undefined%
  \makeatother%
  \begin{picture}(1,0.7255846)%
    \lineheight{1}%
    \setlength\tabcolsep{0pt}%
    \put(0,0){\includegraphics[width=\unitlength,page=1]{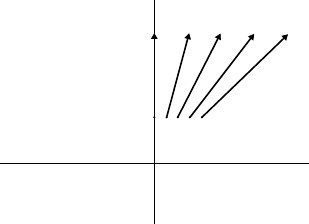}}%
    \put(0.34457625,0.3248185){\makebox(0,0)[lt]{\lineheight{1.25}\smash{\begin{tabular}[t]{l}$(0,\varepsilon)$\end{tabular}}}}%
    \put(0,0){\includegraphics[width=\unitlength,page=2]{non-uniqueness-deltas.pdf}}%
  \end{picture}%
\endgroup%
}
    \caption{We depicted the family $T_t=(e_2+te_1)\delta_{(t\eps,\eps)}$, which solves \eqref{eq:GTE} (in the classical sense) with initial datum $e_2\delta_{(0,\eps)}$. Assuming the Stability property, the limit as $\eps\to 0$ of this family gives rise to the solution $T^2_t=(e_2+te_1)\delta_0$ with initial datum $e_2\delta_0$. On the other hand, building an approximating family from below, the constant family $T^1_t=e_2\delta_0$ also solves the equation starting from the same initial datum. This shows non-uniqueness in the class of natural solutions $\Scal_b$ with finite mass as defined in Definition \ref{def:natural_family}. 
    }
    \label{fig:non-uniqueness}
\end{figure}

On the other hand, we show via an example that if $b$ is merely Lipschitz  uniqueness might fail: 

\begin{proposition}[Non-uniqueness for finite mass currents]
    There exists an autonomous Lipschitz vector field $b:\R^2\to\R^2$ for which the following holds: if $\Scal_b$ is any natural family of solutions then there exist two distinct solutions $(T_t^i)_{t\in[0,1]}\in \Scal_b$, $i=1,2$, starting from the same initial datum.
\end{proposition}

\begin{proof}
    We define the vector field
    \[
    b(x,y):=
    \begin{cases}
    (y,0) & \text{if $y\ge 0$}\\
    0 & \text{otherwise}
    \end{cases}.
    \]
    We claim that if $\Scal_b$ is any natural family of solutions then the following curves of finite mass $1$-currents both belong to $\Scal_b$: 
    \[
    T_t^1:=e_2 \delta_0\qquad T_t^2:=(e_2+te_1)\delta_0.
    \]
    Indeed the following families are solutions because of the consistency assumption (H1):
    \[
    T_t^{1,\eps} :=e_2\delta_{(0,-\eps)},\qquad T_t^{2,\eps}:= (e_2+te_1)\delta_{(t\eps,\eps)}.
    \]
    The first is trivially a solution starting from $e_2\delta_{(0,-\eps)}$, since $b=0$ on the support of $T_t^{1,\eps}$. The second is also a solution starting from $e_2\delta_{(0,\eps)}$: to see this one can consider the $1$-current associated with the segment $\{0\}\times [\eps,\eps+\eta]$ for some $\eta>0$, with orientation $e_2$, and renormalized to have mass 1. Then this is a normal current and the (unique) solution starting from it at time $t=0$ is given by the pushforward with respect to $\Phi_t$. Passing to the limit as $\eta\to 0$ one gets the claim.
    Moreover it is not hard to see that $T_t^{1,\eps}$ and $T_t^{2,\eps}$ converge to $T_t^1$ and $T_t^2$ respectively, thus the latter belong to $\Scal_b$ by the stability assumption (H2). To finish the proof we just observe that at time $t=0$ both solutions coincide with $e_2\delta_0$.
\end{proof}

\bibliographystyle{plain}

\end{document}